\documentclass[10pt,reqno]{amsart}
\usepackage{amssymb,mathrsfs,graphicx,extpfeil,amsmath,bbm}
\usepackage{ifthen,latexsym,float,colortbl}
\usepackage[margin=1in]{geometry}
\usepackage{caption}
\usepackage{arydshln}
\usepackage{sidecap}
\usepackage{rotating,dsfont,stackengine}
\usepackage{refcheck}
\usepackage{colortbl}
\definecolor{black}{rgb}{0.0, 0.0, 0.0}
\definecolor{red}{rgb}{1.0, 0.5, 0.5}
\provideboolean{shownotes} % define a boolean variable
\setboolean{shownotes}{true} % defined as ``true'' or ``false''
\newcommand{\margnote}[1]{
\ifthenelse{\boolean{shownotes}}%
{\marginpar{\raggedright\tiny\texttt{#1}}}%
{}%
}
\newcommand{\hole}[1]{
\ifthenelse{\boolean{shownotes}}%
{\begin{center} \fbox{ \rule {.25cm}{0cm} \rule[-.1cm]{0cm}{.4cm}
\parbox{.85\textwidth}{\begin{center} \texttt{#1}\end{center}} \rule
{.25cm}{0cm}}\end{center}} {} }

\graphicspath{{pics/}}
\title[BGK model for isentropic gas dynamics]{Classical solutions to a BGK-type model relaxing to the isentropic gas dynamics}

\author[Hwang]{Byung-Hoon Hwang}
\address[Byung-Hoon Hwang]{\newline Department of Mathematics Education\newline
Sangmyung University, 20 Hongjimun 2-gil, Jongno-Gu, Seoul 03016, Republic of Korea}
\email{bhhwang@smu.ac.kr}

\numberwithin{equation}{section}

\newtheorem{theorem}{Theorem}[section]
\newtheorem{lemma}{Lemma}[section]

\newtheorem{proposition}{Proposition}[section]

\newcommand{\R}{\mathbb R}

\newcommand{\T}{\mathbb T}

\newcommand{\bq}{\begin{equation}}
\newcommand{\eq}{\end{equation}}
\newcommand{\e}{\varepsilon}
\newcommand{\lt}{\left}
\newcommand{\rt}{\right}
\newcommand{\lal}{\langle}
\newcommand{\ral}{\rangle}
\newcommand{\pa}{\partial}

\newcommand{\intr}{\int_{\R^d}}

\newcommand{\inttr}{\iint_{\T^d \times \R^d}}

\makeatletter
\def\moverlay{\mathpalette\mov@rlay}
\def\mov@rlay#1#2{\leavevmode\vtop{%
   \baselineskip\z@skip \lineskiplimit-\maxdimen
   \ialign{\hfil$\m@th#1##$\hfil\cr#2\crcr}}}
\newcommand{\charfusion}[3][\mathord]{
    #1{\ifx#1\mathop\vphantom{#2}\fi
        \mathpalette\mov@rlay{#2\cr#3}
      }
    \ifx#1\mathop\expandafter\displaylimits\fi}
\makeatother

\makeatother
\makeatletter
\@namedef{subjclassname@2020}{\textup{2020} Mathematics Subject Classification}
\makeatother
\begin{document}
%%%%%%%%%%%%%%%%
\allowdisplaybreaks

\date{\today}



\begin{abstract} 
In this paper, we consider a BGK-type kinetic model relaxing to the isentropic gas dynamics in the hydrodynamic limit. We introduce a linearization of the equation around the global equilibrium. Then we prove the global existence of classical solutions with an exponential convergence rate toward the equilibrium state in the periodic domain when the initial data is a small perturbation of the global equilibrium.

%The BGK model is a model equation of the Boltzmann equation designed to simplify the collision operator for practical purposes.  

%	We establish the global-in-time existence of weak solutions to a variant of the BGK model proposed by Bouchut [J. Stat. Phys., 95, (1999), 113--170] which leads to the barotropic Euler equations in the hydrodynamic limit. Our existence theory makes the quantified estimates of hydrodynamic limit from the BGK-type equations to the multi-dimensional barotropic Euler system discussed by Berthelin and Vasseur [SIAM J. Math. Anal., 36, (2005), 1807--1835] completely rigorous.

\end{abstract}

\maketitle \centerline{\date}

%\tableofcontents

%%%%%%%%%%%%%%%%%%%%%%%%%%%%%%%%%%%%%%%%%%%%%%%%%%%%%%%%%%%%%%%%%%%%%%%%%%%%%%%%%5
%
%
%                        Section: Introduction 
%
%
%%%%%%%%%%%%%%%%%%%%%%%%%%%%%%%%%%%%%%%%%%%%%%%%%%%%%%%%%%%%%%%%%%%%%%%%%%%%%%%%%
\section{Introduction}
 
%The BGK model \cite{BGK54,W54} is a relaxation time approximation of the celebrated Boltzmann equation, which describes the time evolution of velocity distribution functions of rarefied gases based on the relaxation process towards the Maxwellian distribution. One of the important topics in the model is to study the connection between mesoscopic kinetic equations and fluid dynamical macroscopic equations so called  {\it hydrodynamic limits}. By taking into account different asymptotic regimes, it is well-known that the compressible Euler equations and the Navier--Stokes equations can be derived from the BGK model at the formal level by the Chapman--Enskog or Hilbert expansion \cite{BGL91, C88, CIP94, GS03}. In the case of the rigorous derivations, several papers have been reported so far, dealing with the Navier--Stokes--Fourier equations \cite{S03}, the linear incompressible Navier--Stokes equations \cite{B10}, and the nonlinearized compressible Euler equations and the acoustic equations \cite{B04}. For the Boltzmann equation, derivations of incompressible Navier--Stokes equations and Euler equations are established \cite{BGL93, C80, GS04,MS03, Y05}. There are also many different kinetic models, which might not be associated to microscopic descriptions, for conservation laws and balance laws \cite{B99, GM83, LPT94, PT91}. We refer to \cite{Per02, S09, V08} and references therein for the general survey of the hydrodynamic limits of kinetic theory.

\subsection{Model equation}
The BGK model \cite{BGK54,W54} is a relaxation-time approximation of the Boltzmann equation, designed based on the relaxation process toward equilibrium. This model has been used extensively in physics and engineering since it not only yields  satisfactory simulation results but also provides a suitable framework for kinetic modeling. For instance, it has been generalized in various ways to model the dynamics of classical particles \cite{Holway66,Shakhov68}, relativistic particles \cite{AW74,Marle65,Marle69,PR18}, quantum particles \cite{Khalatnikov65}, and gas mixtures \cite{AAP02,BKPY21,BBGSP18,BPS12,HLY23,KPP17}. We also refer to  \cite{LPT94_1, PT91} where the kinetic formulation using the BGK model was applied to derive the kinetic model for scalar conservation laws and related equations. In this paper, we consider a variant of the BGK model introduced in \cite{B99,LPT94_2}, which relaxes to the isentropic gas dynamics in the hydrodynamic limits
\begin{equation}\label{BGK0}
\pa_t F+ v\cdot\nabla_x F  = M[F] - F.
\end{equation}
Here $F\equiv F(x,v,t)$ is the one-particle distribution function at the phase point $(x,v)\in \T^d\times  \R^d$ and time $t\in\R_+$, and $M$ denotes an equilibrium distribution function for isentropic gas dynamics \cite{B99}, given by
\begin{equation}\label{Maxwellian}
M [F]\equiv M(\rho_F,u_F;v)= 
\displaystyle c\left(\frac{2\gamma}{\gamma-1}\rho_F^{\gamma-1}-|v-u_F|^2\right)^{n/2}_+ \qquad \text{for} \quad \displaystyle \gamma \in \lt(1,1+\frac{2}d\rt)
\end{equation}
where $(\cdot)_+$ stands for the positive part of a function  and $\mathbf{1}_A$ the indicator function on
the set $A$. In \eqref{Maxwellian}, $\rho_F$ and $u_F$ denote the macroopic density and bulk velocity of $F$ defined by
\begin{align*}
\rho_F(x,t)&=\int_{\mathbb{R}^d} F(x,v,t)\,dv,\qquad\rho_F(x,t)u_F(x,t)=\int_{\mathbb{R}^d} vF(x,v,t)\,dv
\end{align*}
respectively, and $c_d,n,$ and $c$ are constants given by
\begin{align*}
n=\frac{2}{\gamma-1}-d,
\quad c_d=\frac{d}{|\mathbb{S}_{d}|},\quad \mbox{and} \quad c=\left(\frac{2\gamma}{\gamma-1}\right)^{-\frac{1}{\gamma-1}}\frac{\Gamma\left(\frac{\gamma}{\gamma-1}\right)}{\pi^{\frac d2}\Gamma(\frac n2+1)},
\end{align*}
where $\Gamma$ is the Gamma function. By a direct calculation (see \cite[Appendix]{CH24-weak}), we see that the equilibrium distribution \eqref{Maxwellian} satisfies 
\bq\label{moment_comp}
	\int_{\mathbb{R}^d}(1,v,|v|^2)  M[F]\,dv =
	\left(\rho_F,\rho_Fu_F, 
	\rho_F |u_F|^2 +d \rho_F^\gamma \right)
\eq
The first two relations of \eqref{moment_comp} lead to the local conservation laws of mass and momentum for \eqref{BGK0}:
\begin{align*}\begin{split}
&\pa_t \rho_F+\nabla_x\cdot (\rho_F u_F)=0,\cr 
 &\pa_t (\rho_F u_F)+\nabla_x \cdot (\rho_F u_F\otimes u_F)+\nabla_x\cdot \intr (u_F-v)\otimes (u_F-v)F\,dv=0,
\end{split}\end{align*}
yielding
\begin{equation}\label{conservation laws}
\frac{d}{dt}\inttr F(x,v,t)\,dxdv=0,\qquad \frac{d}{dt}\inttr vF(x,v,t)\,dxdv=0.
\end{equation}
Note that the third relation gives
\[
\nabla_x\cdot\int_{\R^d} ( u_{F} - v)\otimes ( u_{F} - v) M[f]\,dv = C_d \nabla_x \rho^\gamma
\]
which enables us to derive the system of isentropic gas dynamics: 
\begin{align*}
\begin{aligned}
&\pa_t \rho  + \nabla_x \cdot (\rho   u ) = 0,\cr
&\pa_t (\rho   u ) + \nabla_x \cdot (\rho   u  \otimes  u ) + C_d\nabla_x \rho^\gamma  =0
\end{aligned}
\end{align*}
at the formal level, see \cite{CH24-weak}. The kinetic entropy associated with \eqref{BGK0} is given by
\begin{equation*}
H(F,v)= \begin{cases}
\displaystyle \frac{|v|^2}{2}F+\frac{1}{2c^{2/n}}\frac{F^{1+2/n}}{1+2/n} \qquad &\text{for} \quad \displaystyle \gamma \in \lt(1,1+\frac{2}d\rt),\\[4mm]
\displaystyle \frac{|v|^2}{2}F\qquad &\text{for}\quad \displaystyle \gamma=1+\frac{2}{d},
\end{cases}
\end{equation*}
and the following  minimization principle holds: 
$$
\int_{\mathbb{R}^d} H(M[F],v)\,dv \le \int_{\mathbb{R}^d} H(F,v)\,dv,
$$
provided $F+H(F,v)$ belongs to $L^1(\R^d_v)$, see \cite[Section 2.1]{B99} for details.

The BGK-type model \eqref{BGK0} has been widely studied to understand the isentropic gas dynamics at the kinetic level.  The flux vector splitting scheme for general systems of conservation laws was established in \cite{B03}. In the mono-dimensional case, existence and stability of small amplitude travelling wave
solutions were studied in \cite{CHS10}. The weak solutions was established in \cite{BB00} and its hydrodynamic limit to the isentropic gas dynamics was addressed in \cite{BB02,BB02_3}. For the multi-dimensional case, we refer to \cite{CH24-weak} for the existence of weak solutions and \cite{BV05} for the hydrodynamic limit. Also, it was shown in \cite{CH24-hydro} that the isentropic Euler-alignment system can be rigorously derived from \eqref{BGK0} in the presence of the velocity alignment force. Despite the variety of studies above, the large-time behavior of classical solutions has not been discussed yet, whereas for other BGK models it has been intensively studied in the literature, see \cite{BY20,BY23,HRY22,LP19,Yun12,Yun15}. Motivated by this, in this paper, we study the stability of solutions to \eqref{BGK0} around the global equilibrium.

\subsection{Main result}
In this paper, we are concerned with the equation \eqref{BGK0}: 
$$
\pa_t F+ v\cdot\nabla_x F  =M[F]-F
$$
subject to the initial data $ F(x,v,0)= F_0(x,v)$. We define the global equilibrium distribution as
 \begin{equation}\label{global}
 M_0:= M(1,0;v)=c\left(\frac{2\gamma}{\gamma-1}-|v|^2\right)^{n/2}_+,
 \end{equation}
and decompose $F$ into the equilibrium part and the perturbation part:
\begin{equation}\label{decomposition}
 F=M_0+M_0^{\frac{n-2}{2n}} f,\quad\mbox{with}\quad F_0(x,v)=M_0+M_0^{\frac{n-2}{2n}} f_0(x,v).
\end{equation}
Note that here $M_0$ is compactly supported in $v$, so in this regime the velocity support of $F$ is supposed to be entirely included in that of $M_0$, say $v\in \Omega$ (see \cite{CHS10} for the same argument).  Also, we assume $\gamma$ to be strictly less than $1+2/(d+2)$ to make $n-2>0$, which guarantees that the perturbation part is well-defined. Substituting the above decomposition into \eqref{BGK0} and applying the Taylor theorem, one can obtain the linearized equation: 
\begin{align}\label{perturbation}\begin{split}
\pa_t f + v\cdot\nabla_x f  &= L(f)+\Gamma(f)\cr 
f(x,v,0)&=f_0(x,v)
\end{split}\end{align}
 Here $L$ is the linearized operator and $\Gamma$ the nonlinear operator whose  definitions can be found in  \eqref{linearized Cauchy}.  

 The aim of this paper is to establish a global-in-time existence and large-time behavior of solutions to \eqref{perturbation}. For this,  we introduce our energy functional $E$ defined as
\begin{align*}
E(f)(t)&=\sum_{|\alpha|\le N}\| \partial^\alpha f\|^{2}_{L^2_{x,v}}
\end{align*}
where $\alpha$ is the multi-index:
$$
\alpha=[\alpha_0,\alpha_1,\cdots,\alpha_d],\quad\mbox{so that}\quad \pa^\alpha=\pa_t^{\alpha_0}\pa_{x_1}^{\alpha_1}\cdots\pa_{x_d}^{\alpha_d}.
$$
 Throughout this paper, we use the generic constant $C$ which may vary line by line but does not affect the proof of the main result. Then our main result is stated as follows.

\begin{theorem}\label{main_thm}  Let $N\ge 3$ and $\gamma \in (1,1+\frac{2}{4N+6+d}]$. Suppose that the initial data $F_0$ is compactly supported in a way that 
	$$
 F_0=M_0+M_0^{\frac{n-2}{2n}} f_0,
	$$
If $E(f_0)$ is sufficiently small, then there exists a unique global-in-time solution $f$ to \eqref{perturbation} satisfying
\begin{enumerate}
	\item The energy functional is uniformly bounded:
	$$
	E(f)(t)+\int_0^t E(f)(s) ds\le CE(f_0).
	$$
	\item The solution $f$ decays exponentially fast: 
	$$
	\sum_{|\alpha|\le N}\| \partial^\alpha f\|^{2}_{L^2_{x,v}}\le Ce^{-Ct}.
	$$
\end{enumerate}

\end{theorem}

The first step in the proof of Theorem \ref{main_thm} is to linearize the equation  \eqref{BGK0} around the global equilibrium.  For this, we insert \eqref{decomposition} into the equation, consider the transitional distribution between $M_0$ and $M[F]$, and apply the Taylor expansion to $M[F]-F$ up to second order.  In the linearization procedure, we restrict ourselves to the case $\gamma<1+\frac{2}{4+d}$  to ensure that derivatives of the transitional distribution are well-defined. This is also why the possible range of $\gamma$ in the statement of Theorem \ref{main_thm} reduces depending on $N$.  We note that the head term of the perturbation in \eqref{decomposition} is technically necessary to obtain  the coercive property of the linear operator:
$$
\langle L(f),f\rangle_{L^2_v}=-\|\{I-P \}f\|_{L^2_v}^2,
$$
which plays an important role later in the proof of Theorem \ref{main_thm}.  Here since $M_0$ is compactly supported, we need an additional analysis of the support of the transitional distribution before dividing the equation by $M_0^{\frac{n-2}{2n}}$. In fact, we need to show that the velocity support of the transitional distribution is included in that of $M_0$. More details can be found in Section 2. After obtaining the linearized equation \eqref{perturbation}, we apply the nonlinear energy method developed in \cite{Guo02,Guo03,Guo04} to prove the global existence and large-time behavior of solutions.

\noindent\newline

This paper is organized as follows. In Section 2, we introduce the linearization of \eqref{BGK0} around the global equilibrium. In Section 3, we establish the local-in-time existence of solutions to the linearized equation. Finally in Section 4, we establish the global existence and large-time behavior  of solutions.

\section{Linearization around the global equilibrium}
Consider the distribution function 
$$
F\equiv F(x,v,t),\qquad (x,v,t)\in \mathbb{T}^d\times \Omega\times \mathbb{R}_+
$$
verifying the equation \eqref{BGK0} subject to the initial data $ F(x,v,0)= F_0(x,v)$, where $\Omega\subseteq \R^d$ denotes the velocity support of the global equilibrium distribution $M_0$ presented in \eqref{global}. Then it is straightforward from \eqref{BGK0} that the velocity support of $M[f]$ is included in $\Omega$ as well.  We decompose $F$ around the global equilibrium $M_0$ of \eqref{global} as
\begin{equation}\label{decom}
F=M_0+M_0^{\frac{n-2}{2n}} f \qquad \mbox{with}\quad  F_0=M_0+M_0^{\frac{n-2}{2n}} f_0,
\end{equation}
and substitute \eqref{decom} into \eqref{BGK0} to get
\begin{align}\label{iteration22}\begin{split}
 \pa_t f + v\cdot\nabla_x f    &= M_0^{-\frac{n-2}{2n}}\left\{M[F] - M_0\right\}- f,\cr
f(x,v,0)&=f_0(x,v).
\end{split}\end{align}
Letting $n>4$, we obtain the following lemma.
%In the following lemma, we split $M[F]-M_0$ into the linear part and the nonlinear part. 
\begin{lemma}\label{linearization}
Let $F=M_0+M_0^{\frac{n-2}{2n}} f$ and $\gamma\in (1,1+\frac{2}{4+d})$. We then have
\begin{align*}
M_0^{-\frac{n-2}{2n}} \{M[F] - M_0\}&=n\gamma  M_0^{-\frac{n-2}{2n}}\intr M_0^{\frac{n-2}{2n}} f dv+nvM_0^{-\frac{n-2}{2n}}\intr vM_0^{\frac{n-2}{2n}} fdv+\sum_{|\eta|=2}\frac{D^\eta M(\tilde\theta)}{\eta !}(\rho_F-1,u_F)^\eta
\end{align*}
where $M(\theta)$ denotes the transitional distribution defined by
$$
M(\theta)\equiv M(\rho_\theta, u_\theta)\quad \mbox{with}\quad (\rho_\theta, u_\theta)=\theta(\rho_F, u_F)+(1-\theta)(1, 0)\ ~\mbox{for}\ ~ \theta\in[0,1],
$$
and $\tilde\theta\in (0,1)$ dependent of $t$ and $x$. 
\end{lemma}
\begin{proof}
By definition, $M(\theta)$ is given by 
$$
M(\theta)=c\left(\frac{2\gamma}{\gamma-1}\rho_\theta^{\gamma-1}-|v-u_\theta|^2\right)_+^{n/2} \qquad \mbox{for all }v\in \Omega.
$$
Thus, it follows from  the Taylor expansion that there exists $\tilde\theta\in(0,1)$ dependent of $t$ and $x$ satisfying
\begin{align*} 
M[F]-M_0&\equiv M(1)-M(0)\cr
&=\nabla_{\rho_\theta,u_\theta}M(\theta)\cdot (\rho_F-1,u_F)\big|_{\theta=0}+\sum_{|\eta|=2}\frac{D^\eta M(\tilde\theta)}{\eta !}(\rho_F-1,u_F)^\eta
\end{align*}
 For all the derivatives of $M(\theta)$ to be well-defined, we set $n> 4$, which corresponds to $\gamma \in (1,1+\frac{2}{4+d})$. A simple calculation gives
$$
\partial_{\rho_\theta}M(\theta)=n\gamma \rho_\theta^{\gamma-2} M^{\frac{n-2}{n}}(\theta),\quad
\nabla_{u_\theta}M(\theta)=n(v-u_\theta) M^{\frac{n-2}{n}}(\theta),
$$
and
$$
\rho_F-1=\intr M_0^{\frac{n-2}{2n}} f\,dv,\quad u_F=\intr vM_0^{\frac{n-2}{2n}} f\,dv.
$$
Thus we have
\begin{align}\label{linearized}
M[F]-M_0&=n\gamma  M^{\frac{n-2}{n}}_0\intr M_0^{\frac{n-2}{2n}} f \,dv+nvM^{\frac{n-2}{n}}_0\intr vM_0^{\frac{n-2}{2n}} f\,dv+\sum_{|\eta|=2}\frac{D^\eta M(\tilde\theta)}{\eta !}(\rho_F-1,u_F)^\eta
\end{align}
for some $\tilde\theta\in (0,1)$ dependent of $t$ and $x$. Now it only remains to show that  $supp(M(\tilde\theta))\subseteq supp(M_0)$ for all $\tilde\theta\in (0,1)$ in order to divide the above relation by $M_0^{-\frac{n-2}{2n}}$. We see that
\begin{align*} 
\pa_{\rho_\theta}^2M(\theta)&=n\gamma(\gamma-2)\rho_\theta^{\gamma-3}M^{\frac{n-2}{n}}(\theta)+n(n-2)\gamma^2\rho_\theta^{2\gamma-4}M^{\frac{n-4}{n}}(\theta),\cr
\pa_{\rho_\theta}\pa_{u_{\theta i}} M(\theta)&= n(n-2)\gamma \rho_\theta^{\gamma-2}(v_i-u_{\theta i})M^{\frac{n-4}{n}}(\theta),\cr
\pa_{u_{\theta i}}^2M(\theta)&=-nM^{\frac{n-2}{n}}(\theta)+n(n-2)(v_i-u_{\theta i})^2 M^{\frac{n-4}{n}}(\theta),\cr
\pa_{u_{\theta i}}\pa_{u_{\theta j}}M(\theta)&=n(n-2)(v_i-u_{\theta i})(v_j-u_{\theta j})M^{\frac{n-4}{n}}(\theta),
 \end{align*}
which gives
\begin{align*} 
\sum_{|\eta|=2}\frac{D^\eta M(\tilde\theta)}{\eta !}(\rho_F-1,u_F)^\eta&=-\frac n2\left(\gamma(2-\gamma)(\rho_F-1)^2 \rho_{\tilde\theta}^{\gamma-3}+|u_F|^2 \right)M^{\frac{n-2}{n}}(\tilde\theta)\cr 
&+\frac{n(n-2)}{2} \left\{\gamma (\rho_F-1)\rho_{\tilde\theta}^{\gamma-2} +\sum_{i=1}^3u_{F i}(v_i-u_{\tilde\theta i})\right\}^2M^{\frac{n-4}{n}}(\tilde \theta)\cr 
&=:-I_1+I_2.
 \end{align*}
Since $supp(M[F]\subseteq supp(M_0)$, it is straightforward from \eqref{linearized} that $-I_1+I_2$ vanishes for every $v\notin supp(M_0)$. On the other hand, $I_1$ and $ I_2$ are non-negative in case of  $\gamma\in (1,1+\frac{2}{4+d})$, i.e. $n>4$. Also, it is easy to check that for each $t$ and $x$, $M^{\frac{n-2}{n}}$ and $I_2$ are linearly independent in $v$. Therefore one can see that $-I_1+I_2=0$  only when $(\rho_F,u_F)=(1,0)$ or $M(\tilde\theta)=0$. From both cases, we conclude that  $supp(M(\tilde\theta))\subseteq supp(M_0)$. 
 \end{proof}
Applying Lemma \ref{linearization} to \eqref{iteration22}, one finds
\begin{align}\label{iteration2222}\begin{split}
\pa_t f + v\cdot\nabla_x f  &=n\gamma  M^{\frac{n-2}{2n}}_0\intr M_0^{\frac{n-2}{2n}} f \,dv+nvM^{\frac{n-2}{2n}}_0\intr vM_0^{\frac{n-2}{2n}} f\,dv- f\cr
&+M_0^{-\frac{n-2}{2n}}\sum_{|\eta|=2}\frac{D^\eta M(\tilde\theta)}{\eta !}(\rho_F-1,u_F)^\eta.
\end{split}\end{align}
On the right-hand side, the first three terms are linear parts and the last term  is a nonlinear part. To analyze in more detail, we let $\mathcal{N}$ be a $(d+1)$-dimensional space spanned by $\displaystyle\big\{M_0^{\frac{n-2}{2n}},vM_0^{\frac{n-2}{2n}} \big\}$ and denote $e_i$ ($i=1,\cdots,d+1$) by
$$
e_1=\left\{n\gamma \right\}^{\frac 12} M_0^{\frac{n-2}{2n}},\qquad e_{1+j}=n^{\frac 12}vM_0^{\frac{n-2}{2n}}\ (j=1,\cdots, d).
$$
\begin{lemma}\label{orthonormal}
	 $\displaystyle\{e_i\}_{i=1,\cdots,d+1}$ forms an orthonormal basis of $\mathcal{N}$ in $L^2_v$, i.e.
	 $$
	 \lal e_i,e_j\ral_{L^2_v}=\begin{cases}
	 1 &\mbox{if }i\neq j,\cr
	 0 &\mbox{otherwise,}
	 \end{cases} 
	 $$
	 where $\lal \cdot,\cdot,\ral_{L^2_v}$ denotes the usual $L^2$-inner product with respect to $v\in \R^d$.
\end{lemma}
\begin{proof}
	Due to the oddness, it suffices to show that $ \lal e_i,e_i\ral_{L^2_v}=1$. For $i=1$, we use the change of variables and spherical coordinates to see  
	\begin{align*}
	\int_{\mathbb{R}^d}M_0^{\frac{n-2}{n}}\,dv &=c	\int_{\mathbb{R}^d}\left(\frac{2\gamma}{\gamma-1} -|v|^2\right)^{\frac {n-2}{2}} \mathbf{1}_{|v|^2\le \frac{2\gamma}{\gamma-1}}\,dv\cr
	&= c\left(\frac{2\gamma}{\gamma-1} \right)^{\frac{n+d}{2}-1}	\int_{\mathbb{R}^d}\left(1-|v|^2\right)^{\frac{n-2}{2}} \mathbf{1}_{|v|^2\le 1}\,dv\cr
	&= c\left(\frac{2\gamma}{\gamma-1}\right)^{\frac{n+d}{2}-1}	\left|\mathbb{S}_{d-1}\right|\int_0^1\left(1-r^2\right)^{\frac {n-2}{2}} r^{d-1}\,dr.
	\end{align*}
	It follows from the definition of $c$ and $n$ that
	\begin{align*}
	\int_{\mathbb{R}^d}M_0^{\frac{n-2}{n}}\,dv &= \frac{\Gamma\left(\frac{\gamma}{\gamma-1}\right)}{\pi^{\frac d2}\Gamma(\frac n2+1)}\left(\frac{2\gamma}{\gamma-1}\right)^{-1}	\left|\mathbb{S}_{d-1}\right|\int_0^1\left(1-r^2\right)^{\frac{n-2}{2}} r^{d-1}\,dr\cr
	&= 2 \frac{\Gamma\left(\frac{\gamma}{\gamma-1}\right)}{\Gamma(\frac n2+1)\Gamma\left(\frac{d}{2}\right)}\left(\frac{2\gamma}{\gamma-1}\right)^{-1}	\int_0^1\left(1-r^2\right)^{\frac {n-2}{2}} r^{d-1}\,dr.
	\end{align*}
Applying the change of variables again, one finds
	\begin{align*}
	\int_{\mathbb{R}^d}M_0^{\frac{n-2}{n}}\,dv &= \frac{\Gamma\left(\frac{\gamma}{\gamma-1}\right)}{\Gamma(\frac n2+1)\Gamma\left(\frac{d}{2}\right)}	\left(\frac{2\gamma}{\gamma-1}\right)^{-1}\int_0^1\left(1-r\right)^{\frac {n-2}{2}} r^{\frac d2-1}\,dr\cr
	& = \frac{\Gamma\left(\frac{\gamma}{\gamma-1}\right)}{\Gamma(\frac n2+1)\Gamma\left(\frac{d}{2}\right)}\left(\frac{2\gamma}{\gamma-1}\right)^{-1}	B\left(\frac d2,\frac n2\right),
	\end{align*}
where $B(\cdot,\cdot)$ denotes the Beta function:
	$$
	B\left(p,q\right):=\int_0^1\left(1-r\right)^{q-1} r^{p-1}\,dr.
	$$
It is well known that the beta function is related to the ratio between the Gamma functions as
	$$
	B(p,q)=\frac{\Gamma(p)\Gamma(q)}{\Gamma(p+q)},
	$$
	and the Gamma function satisfies
	$$
	\frac{\Gamma(r+1)}{\Gamma(r)}=r.
	$$
	Applying these relations together with the definition of $n$, we obtain
	\begin{align*}
	\int_{\mathbb{R}^d}M_0^{\frac{n-2}{n}}\,dv 	&=\frac{\Gamma\left(\frac{\gamma}{\gamma-1}\right)}{\Gamma(\frac n2+1)\Gamma\left(\frac{d}{2}\right)}\left(\frac{2\gamma}{\gamma-1}\right)^{-1}	\frac{\Gamma\left(\frac d2\right)\Gamma\left(\frac n2\right)}{\Gamma\left(\frac{d+n}{2}\right)}=\frac{1}{n\gamma},
	\end{align*}
which implies
	$$
	 \lal e_1,e_1\ral_{L^2_v}=  n\gamma\intr  M_0^{\frac{n-2}{n}} \,dv=1.
	$$
In a similar way, we have
	\begin{align*} 
	\int_{\mathbb{R}^d}v_j^2M_0^{\frac{n-2}{n}}\,dv 
	& = c\int_{\mathbb{R}^d}v_j^2\left(\frac{2\gamma}{\gamma-1}-|v|^2\right)^{\frac {n-2}{2}} \mathbf{1}_{|v|^2\le \frac{2\gamma}{\gamma-1}}\,dv\cr
	& = \frac cd\int_{\mathbb{R}^d}|v|^2\left(\frac{2\gamma}{\gamma-1}-|v|^2\right)^{\frac {n-2}{2}} \mathbf{1}_{|v|^2\le \frac{2\gamma}{\gamma-1}}\,dv\cr
&= \frac 1d\frac{\Gamma\left(\frac{\gamma}{\gamma-1}\right)}{\pi^{\frac d2}\Gamma(\frac n2+1)}	\left|\mathbb{S}_{d-1}\right|\int_0^1\left(1-r^2\right)^{\frac{n-2}{2}} r^{d+1}\,dr\cr
	&=\frac 1d\frac{\Gamma\left(\frac{\gamma}{\gamma-1}\right)}{\Gamma(\frac n2+1)\Gamma\left(\frac{d}{2}\right)}B\left(\frac{d}{2}+1,\frac n2\right).
	\end{align*}
	Note that the Beta function satisfies    
	$$
	B(x+1,y)=\frac{x}{x+y}B(x,y)
	$$
	for positive real numbers $x$ and $y$. Applying this, one finds
	\begin{align*} 
	\int_{\mathbb{R}^d}v_j^2M_0^{\frac{n-2}{n}}\,dv &=\frac 1d\frac{\Gamma\left(\frac{\gamma}{\gamma-1}\right)}{\Gamma(\frac n2+1)\Gamma\left(\frac{d}{2}\right)}\left\{\frac{d}{d+n}B\left(\frac{d}{2},\frac n2\right)\right\}\cr
	&=\frac 1d\frac{\Gamma\left(\frac{\gamma}{\gamma-1}\right)}{\Gamma(\frac n2+1)\Gamma\left(\frac{d}{2}\right)}\left\{\frac{d}{d+n}\frac{\Gamma\left(\frac d2\right)\Gamma\left(\frac n2\right)}{\Gamma\left(\frac{1}{\gamma-1}\right)}\right\}\cr
	&=\frac{1}{n},
 		\end{align*}
	which gives 
	$$
	 \lal e_{1+j},e_{1+j}\ral_{L^2_v}=  n\intr v_{j}^2  M_0^{\frac{n-2}{n}} \,dv=1\quad \mbox{for all } j=1,\cdots,d.
	$$
\end{proof}
From Lemma \ref{orthonormal}, we define the orthonormal projection $P$ from $L^2_v$ onto $\mathcal{N}:$ \begin{align*}
P(f):=\sum_{i=1}^{d+1} \lal f,e_i\ral_{L^2_v} e_i.
\end{align*}
Using this, \eqref{iteration2222} can be rewritten as
\begin{align}\label{linearized Cauchy}\begin{split}
\pa_t f + v\cdot\nabla_x f  &= L(f)+\Gamma(f),\cr
f(x,v,0)&= f_0(x,v),
\end{split}\end{align}
where  $L:=P-I$ is the linear operator, and $\Gamma$ is the nonlinear operator given by
\begin{align}\label{Gamma}\begin{split} 
\Gamma(f)&= -\frac n2\left(\gamma(2-\gamma)(\rho_F-1)^2 \rho_{\tilde\theta}^{\gamma-3}+|u_F|^2 \right)M^{\frac{n-2}{n}}(\tilde\theta)M_0^{-\frac{n-2}{2n}} \cr 
&+\frac{n(n-2)}{2} \left\{\gamma (\rho_F-1)\rho_{\tilde\theta}^{\gamma-2} +\sum_{i=1}^3u_{F i}(v_i-u_{\tilde\theta i})\right\}^2M^{\frac{n-4}{n}}(\tilde \theta) M_0^{-\frac{n-2}{2n}}
\end{split}\end{align} 
where $\tilde\theta$ is dependent of $t$ and $x$. In the following lemma, we present the properties of the linear operator $L$  for later use.
\begin{proposition}\label{linear}
The linear operator $L$ satisfies 
\begin{enumerate}
	\item $\displaystyle Ker(L)=\mathcal{N},$
	\item $\displaystyle\lal L(f),f\ral_{L^2_v}=-\|\{I-P\}f \|^2_{L^2_v},$
\end{enumerate}
\end{proposition}
	 \begin{proof}
The proof can be directly obtained by Lemma \ref{orthonormal}, 
%so we only prove the last one.  Recall that
%	\begin{align*}
%M[F]-M_0=M_0^{\frac{n-2}{2n}}\{L(f)+\Gamma(f) \}.
%\end{align*}
%We then have
%\begin{align*}
%\lal\Gamma(f),e_1\ral_{L^2_v}&=\lal M_0^{-\frac{n-2}{2n}}(M[F]-M_0),e_1\ral_{L^2_v}-\lal L(f),e_1\ral_{L^2_v}\cr
%&=	\{n\gamma\}^{\frac 12}\intr M[f]-M_0 \,dv-\lal P(f)-f,e_1\ral_{L^2_v}\cr 
%&=0,
%\end{align*}
%where we used \eqref{moment_comp} and Lemma \ref{orthonormal}. In the same manner, one can show the other cases, 
so we omit it.
	 	
	 	\end{proof}

\section{Local-in-time existence}
In this section, we prove the local-in-time existence of smooth solutions to \eqref{linearized Cauchy}. For this, we start with the $L^2$-estimates of  the nonlinear operator $\Gamma$.
\begin{lemma}\label{nonlinear}
Let $N\ge 3$ and $\gamma \in (1,1+\frac{2}{4N+6+d}]$. If $E(f)(t)$ is sufficiently small,	then $\Gamma(f)$ satisfies
\begin{enumerate}
	\item[(1)] $\displaystyle \|\pa^\alpha \Gamma(f)\|_{L^2_v}\le C\sum_{|\beta|\le |\alpha|}\|\pa^{\beta}f\|_{L^2_v}\|\pa^{\alpha-\beta}f\|_{L^2_v}, $
	\item[(2)]$\displaystyle
	\lal \pa^\alpha \Gamma (f),h\ral_{L^2_v}\le C\sum_{|\beta|\le |\alpha|}\|\pa^{\beta}f\|_{L^2_v}\|\pa^{\alpha-\beta}f\|_{L^2_v}\|h\|_{L^2_v}.
	$
\end{enumerate}
	
\end{lemma}
\begin{proof} Since (2) directly follows from (1), it suffices to prove (1) only. Recall from \eqref{Gamma} that  
\begin{align*}
\Gamma(f)&=-\frac n2\left(\gamma(2-\gamma)(\rho_F-1)^2 \rho_{\tilde\theta}^{\gamma-3}+|u_F|^2 \right)M^{\frac{n-2}{n}}(\tilde\theta)M_0^{-\frac{n-2}{2n}}\cr 
&+\frac{n(n-2)}{2} \left\{\gamma (\rho_F-1)\rho_{\tilde\theta}^{\gamma-2} +\sum_{i=1}^3u_{F i}(v_i-u_{\tilde\theta i})\right\}^2M^{\frac{n-4}{n}}(\tilde \theta)M_0^{-\frac{n-2}{2n}}
\end{align*}
where $supp(M(\tilde\theta))\subseteq supp(M_0)$.  Observe that  under the assumption $E(f)(t)\ll 1$, we have 
\begin{equation}\label{rho theta}
|\rho_\theta-1|=\theta\left| \intr M_0^{\frac{n-2}{2n}} f\,dv \right|\le C \|f\|_{L^2_v} \le C\| \|f\|_{L^2_v}\|_{L^\infty_x}\le C\{E(f)\}^{\frac 12}\ll 1,
\end{equation} 
and
\begin{equation}\label{u theta}
|u_{\theta}|=\theta\left|\frac{1}{\rho_F}\intr vM_0^{\frac{n-2}{2n}} f\,dv\right|\le \frac{\left|\intr vM_0^{\frac{n-2}{2n}} f\,dv\right|}{1+\intr M_0^{\frac{n-2}{2n}} f\,dv} \le C\|f\|_{L^2_v} \ll 1
\end{equation}
where we used H\"{o}lder's inequality and the Sobolev embedding $H^2(\T^3)\subseteq L^\infty(\T^3)$
Obviously, it holds true for $\rho_F-1$ and $u_F$ as well by definition. In the same manner,  one finds
$$
|\pa^\alpha \{\rho_F-1\}|+|\pa^\alpha u_F|+|\pa^\alpha\rho_\theta|+|\pa^\alpha u_\theta| \le C\|\pa^\alpha f\|_{L^2_v}.
$$
From these observations, one can see that
\begin{align*}\begin{split}
|\pa^\alpha \Gamma(f)|&\le C(1+|v|^2)M_0^{-\frac{n-2}{2n}}\sum_{\substack{|\beta_1|+|\beta_2|\cr +|\beta_3|\le|\alpha|}} \|\partial^{\beta_1}f\|_{L^2_v}\|\partial^{\beta_2}f\|_{L^2_v}\left|\pa^{\beta_3}\{M^{\frac{n-2}{n}}(\tilde\theta)+M^{\frac{n-4}{n}}(\tilde\theta)\}\right|\cr 
&\le CM_0^{-\frac{n-2}{2n}}\sum_{\substack{|\beta_1|+|\beta_2|\cr +|\beta_3|\le|\alpha|}} \|\partial^{\beta_1}f\|_{L^2_v}\|\partial^{\beta_2}f\|_{L^2_v} \mathcal{P}\left(\|f\|_{L^2_v},\cdots,\|\pa^{\beta_3}f\|_{L^2_v}\right)\left(M^{\frac{n-2}{n}}(\tilde\theta)+ M^{\frac{n-4-2|\beta_3|}{n}}(\tilde\theta) \right),
\end{split}\end{align*}
for some generic polynomial $\mathcal{P}$. Noticing that $supp(M(\tilde\theta))\subseteq supp(M_0)$, $|\beta_3|\le N$, and $M(\theta)$ are uniformly bounded due to \eqref{rho theta} and \eqref{u theta},   we get
$$
M_0^{-\frac{n-2}{2n}}\left(M^{\frac{n-2}{n}}(\tilde\theta)+ M^{\frac{n-4-2|\beta_2|}{n}}(\tilde\theta) \right)< \infty
$$
for $n\ge 4N+6$, which corresponds to $\gamma \in (1,1+\frac{2}{4N+6+d}]$. Finally, applying the Sobolev embedding $H^2(\R^3)\subseteq L^\infty(\R^3)$  again to lower order derivatives, we obtain the desired result.
\end{proof}

\begin{lemma}\label{nonlinear2}
	Let $N\ge 3$ and  $\gamma \in (1,1+\frac{2}{6+d}]$. If $E(f)(t)$ and $E(g)(t)$ are sufficiently small,	we then have
$$
\left\lal \Gamma(f)-\Gamma(g), h\right\ral_{L^2_v}\le C\|f-g\|_{L^2_v}\|h \|_{L^2_v}.
$$
\end{lemma}
\begin{proof}
The proof can be obtained by the triangle inequality and  almost the  same methodology as in the proof of Lemma \ref{nonlinear2}, so we omit it to avoid repetition.
\end{proof}
To construct the local solution, we consider the approximation sequence $\{f^k\}$  of \eqref{linearized Cauchy} as 
\begin{align}\label{approximate}\begin{split}
\pa_t f^{k+1} + v\cdot\nabla_x f^{k+1}  &= P(f^k)- f^{k+1}+\Gamma(f^k) 
\end{split}\end{align}
with the initial data and first iteration step:
$$
f^k(x,v,0)=f_0(x,v)\quad\mbox{for all }k\ge 1 \quad\mbox{and}\quad f^0(x,v,t)=f_0(x,v).
$$

\begin{lemma}\label{inductive}
	There exist positive constants $\eta$ and $T_*$ such that if $E(f_0)\le \frac \eta 2$, then $E(f^k)(t)\le \eta$ implies $E(f^{k+1})(t)\le \eta$  for all $  t\in [0,T_*]$.
\end{lemma}
\begin{proof}
Applying $\pa^\alpha$ to \eqref{approximate}, taking $L^2$-inner product with $\pa^\alpha f^{k+1}$, and summing over $|\alpha|<N$, we have
	\begin{align*}
\frac 12	\frac{d}{dt} \sum_{|\alpha|\le N}\|\pa^\alpha f^{k+1}\|^2_{L^2_{x,v}}+\sum_{|\alpha|\le N}\|\pa^\alpha f^{k+1}\|^2_{L^2_{x,v}}&=\sum_{|\alpha|\le N}\lal P(\pa^\alpha f^k), \pa^\alpha f^{k+1}\ral_{L^2_{x,v}}+\sum_{|\alpha|\le N}\lal \pa^\alpha \Gamma (f^k), \pa^\alpha f^{k+1}\ral_{L^2_{x,v}}\cr
&=:I_1+I_2
	\end{align*}
	where 
\begin{align}\label{I1} \begin{split}
I_1&=\sum_{|\alpha|\le N}\int_{\R^3} \sum_{i=1}^4\lal \pa^\alpha f^{k}, e_i \ral_{L^2_v}\lal \pa^\alpha f^{k+1}, e_i \ral_{L^2_v}  \,dx\cr
&\le C\sum_{|\alpha|\le N}\int_{\R^3}  \|\pa^\alpha f^{k}\|_{L^2_v}\|\pa^\alpha f^{k+1}\|_{L^2_v}\,dx\cr
&\le C_\e\sum_{|\alpha|\le N}\|\pa^\alpha f^k\|^2_{L^2_{x,v}}+\e\sum_{|\alpha|\le N}\|\pa^\alpha f^{k+1}\|^2_{L^2_{x,v}}.
\end{split}\end{align}
Also, it follows from Lemma \ref{nonlinear} that
\begin{align*}
I_2\le  C\sum_{\substack{|\alpha_1|+|\alpha_2|\le  N}}\|\pa^{\alpha_1}f^k\|_{L^2_v}\|\pa^{\alpha_2}f^k\|_{L^2_v}\|\pa^\alpha f^{k+1}\|_{L^2_v}.
\end{align*}
Applying $H^2(\T^3 )\subseteq L^\infty(\T^3 )$ to the lower order derivatives and using Young's inequality gives
\begin{align*}
I_2\le  C\{E(f^k)(t)\}^{\frac 12}\sum_{|\alpha|\le N}\left( \|\pa^{\alpha}f^k\|^2_{L^2_v}+\|\pa^\alpha f^{k+1}\|_{L^2_v}^2\right).
\end{align*}
Combining these estimates, we get
\begin{align}\label{E estimate}
\frac 12 \frac {d}{dt}E(f^{k+1})(t)+(1-\e-C\eta^{\frac 12})E(f^{k+1})(t)\le C_\e E(f^k)(t)+ C\{E(f^k)(t)\}^{\frac 32}.
\end{align}
Taking $\e$ small enough and integrating over $t\in [0,T_*]$, one finds
\begin{align*}
E(f^{k+1})(t)&\le e^{-C_\eta t}E(f^0)+C\int_0^t e^{-C_\eta (t-s)} E(f^k)(s)+ \{E(f^k)(s)\}^{\frac 32}\,ds\cr
&\le \frac \eta 2+CT_*\eta.
\end{align*}
Finally, taking $T_*$ small enough we obtain the desired result.
	\end{proof}
Now we are ready to prove the existence of unique local solution $f(x,v,t)$ to \eqref{linearized Cauchy}. 

\begin{proposition}\label{local}
	Let $N\ge 3$, $\gamma\in (1,1+\frac{2}{4N+6+d}]$. Assume that  $F_0=M_0+M_0^{\frac{n-2}{2n}}f_0$ is non-negative and satisfies 
	$$
\iint_{\T^d\times\R^d}M_0^{\frac{n-2}{2n}}f_0 \,dxdv=0,\quad \iint_{\T^d\times\R^d}vM_0^{\frac{n-2}{2n}}f_0 \,dxdv=0.
$$
	Then there exist $\eta$ and $T_*$ such that if  $E(f_0)\le \frac\eta 2$ and $T_*\le \frac \eta 2$, \eqref{linearized Cauchy} admits a local-in-time existence of a unique solution $f(x,v,t)$ satisfying that for $t\in [0,T_*)$,
	\begin{enumerate}
		\item[(1)]  The distribution function $F(x,v,t)$ is non-negative:
		$$
		F(x,v,t)=M_0+M_0^{\frac{n-2}{2n}}f \ge 0.
		$$ 
		\item[(2)]  The energy functional $E(f)(t)$ is continuous and uniformly bounded:
		$$
		\sup_{0\le t\le T_*}E(f)(t)\le \eta.
		$$
		\item[(3)] The total mass and momentum of perturbation parts are equal to zero:
		$$
		\iint_{\T^d\times\R^d}M_0^{\frac{n-2}{2n}}f \,dxdv=0,\quad \iint_{\T^d\times\R^d}vM_0^{\frac{n-2}{2n}}f \,dxdv=0.
		$$
	\end{enumerate}	  
\end{proposition}
\begin{proof}
Observe from \eqref{approximate} that
$$
\frac{d}{dt}\{f^{k+1}-f^k\}+v\cdot\nabla_x\{f^{k+1}-f^k\}+f^{k+1}-f^k=P(f^k)-P(f^{k-1})+\Gamma(f^k)-\Gamma(f^{k-1})
$$
Multiplying the above by $f^{k+1}-f^k$ and integrating over $(x,v)\in\T^3\times\R^3$, we have
		\begin{align*}
	&\frac 12	\frac{d}{dt} \|f^{k+1}-f^{k}\|^2_{L^2_{x,v}}+\|f^{k+1}-f^{k}\|^2_{L^2_{x,v}}\cr
	&=\left\lal P(f^k-f^{k-1}),  f^{k+1}-f^k\right\ral_{L^2_{x,v}}+\left\lal  \Gamma (f^k)- \Gamma (f^{k-1}), f^{k+1}-f^k\right\ral_{L^2_{x,v}}\cr
	&=:I_1+I_2.
		\end{align*}
	In the same manner as \eqref{I1}, one finds
\begin{align*}
I_1\le C_\e\|f^{k}-f^{k-1}\|^2_{L^2_{x,v}}+\e\|f^{k+1}-f^k\|^2_{L^2_{x,v}}
\end{align*}
and applying Lemma \ref{nonlinear2} gives
\begin{align*}
I_2\le  C\eta^{\frac 12}\left( \|f^{k}-f^{k-1}\|^2_{L^2_{x,v}}+\|f^{k+1}-f^k\|_{L^2_{x,v}}^2\right).
\end{align*}
Therefore, for sufficiently small $\eta$ and $\e$ we get
	\begin{align*}
\frac{d}{dt} \|f^{k+1}-f^{k}\|^2_{L^2_{x,v}}+C_{\e,\eta}\|f^{k+1}-f^{k}\|^2_{L^2_{x,v}}\le C \|f^{k}-f^{k-1}\|^2_{L^2_{x,v}}
\end{align*}
which implies that $\{f^k\}$ is Cauchy in $L^\infty(0,T_*; L^2(\T^3\times\R^3))$. Therefore we obtain the local-in-time existence of a unique solution $f$ and the uniform bound of the energy functional due to Lemma \ref{inductive}. To show the continuity of $E(f)(t)$, we observe from \eqref{E estimate}  that
\begin{align*}
\frac{d}{dt}E(f)(t)\le C\left( E(f)(t)+ \{E(f)(t)\}^{\frac 32}\right).
\end{align*}
Integrating over $[s,t]$, one finds
$$
|E(f)(t)-E(f)(s)|\le C\int_s^t  E(f)(\tau)+ \{E(f)(\tau)\}^{\frac 32} \,d \tau
$$
which combined with Lemma \ref{inductive} gives the desired result. The positivity of $F$ can be obtained by the iteration \eqref{approximate} with the initial assumption $F_0\ge 0$.  Finally due to the smoothness of the solution $f$, it follows from \eqref{conservation laws} and the initial condition that 
	$$
\iint_{\T^d\times\R^d}M_0^{\frac{n-2}{2n}}f \,dxdv=\iint_{\T^d\times\R^d}M_0^{\frac{n-2}{2n}}f_0 \,dxdv=0,
$$
and
$$
 \iint_{\T^d\times\R^d}vM_0^{\frac{n-2}{2n}}f \,dxdv= \iint_{\T^d\times\R^d}vM_0^{\frac{n-2}{2n}}f_0 \,dxdv=0,
$$
 which completes the proof.
\end{proof}

\section{Proof of Theorem \ref{main_thm}: global existence and large-time behavior}
\subsection{Macro-micro decomposition}
The aim of this subsection is to study the estimate of $Pf$ in terms of $\{I-P\}f$, which combined with the coercive property of the linear operator $L$ (see Proposition \ref{linear}) enables us to prove the global existence and large-time behavior of solutions.  Recall that the projection operator $P(f)$ takes the form of
$$
P(f)=\left\{a(x,t)+v\cdot b(x,t)  \right\}M_0^{\frac{n-2}{2n}}
$$
where 
$$
a(x,t)=n\gamma \intr fM_0^{\frac{n-2}{2n}}\,dv,\quad b(x,t)=n \intr v fM_0^{\frac{n-2}{2n}}\,dv.
$$
Substituting $f=P(f)+\{I-P\}f$ into \eqref{linearized Cauchy}, we get
\begin{align*}
\{\partial_t+v\cdot\nabla_x \}P(f)&=\{-\partial_t -v\cdot\nabla_x+L \}\{I-P \}(f)+\Gamma(f)\cr
&=: \ell\{I-P\}(f)+h(f),
\end{align*}  
and a direct calculation gives
\begin{align}\label{mM}
\left\{ \pa_t a+\sum_{i=1}^dv_i(\pa_{x_i}a+\pa_t b)+\sum_{i=1}^dv_i\pa_{x_i}\sum_{j=1}^d v_j b_j \right\}M_0^{\frac{n-2}{2n}}=\ell\{I-P\}(f)+h(f).
\end{align}
Define
\begin{equation}\label{micromacro basis}
\left\{e_{a},e_{a_i},e_{b_{ij}} \right\}:=\left\{ 1, v_i,v_iv_j \right\}M_0^{\frac{n-2}{2n}}.
\end{equation}
Using this, then \eqref{mM} can be written as 
\begin{enumerate}
	\item $\pa_ta=\ell_{a}+h_{a}$,
	\item $\pa_{x_{i}}a+\partial_{t}b_{i}=\ell_{a_i}+h_{a_i}$,
	\item $\partial_{x_{i}}b_{j}+\partial_{x_{j}}b_{i}=\ell_{b_{ij}}+h_{b_{ij}}$
\end{enumerate}
where $\delta_{ij}$ is the Kronecker delta, and $\ell_{(\cdot)}$ and $h_{(\cdot)}$ denote the inner product of $\ell\{I-P\}(f)$ and $h(f)$ with the corresponding basis \eqref{micromacro basis}. This is often called the macro-micro decomposition. Now we follow the standard argument introduced in \cite{Guo02,Guo03,Guo04} to obtain the full coercivity of $L$. The proof will be similar to that of \cite{Guo02,Guo03,Guo04} but for the completeness of the paper, we prove it in detail.
\begin{lemma}\label{ab estimates}
Let $f$ be a local solution constructed in Proposition \ref{local}. We then have
	\begin{align*}
\sum_{|\alpha|\le N} \big( \|\pa^\alpha a\|^2_{L^2_x}+\| \pa^\alpha b\|^2_{L^2_x}\big) &\le 	C\sum_{|\alpha|\le N-1}\biggl(\|\pa^\alpha\ell_a\|^2_{L^2_x}+\sum_{i=1}^d\|\pa^\alpha\ell_{a_i}\|^2_{L^2_x}+\sum_{i,j=1}^d\|\pa^\alpha\ell_{b_{ij}}\|^2_{L^2_x}\biggl) 
\cr
& +C	\sum_{|\alpha|\le N-1}\biggl(\|\pa^\alpha h_a\|^2_{L^2_x}+\sum_{i=1}^d\|\pa^\alpha h_{a_i}\|^2_{L^2_x}+\sum_{i,j=1}^d\|\pa^\alpha h_{b_{ij}}\|^2_{L^2_x}\biggl)
\end{align*}
 
	\end{lemma}
\begin{proof}
 Observe that 
	\begin{align*}
	\Delta_x \partial^{\alpha}b_{i}&=\sum_{j=i}^d\pa_{x_j}^2 \pa^\alpha b_i \cr
	&=\sum_{j=1}^d\{-\partial_{x_j}\pa_{x_i}\partial^{\alpha}b_j+\partial_{x_j}\partial^{\alpha}(\ell_{b_{ij}}+h_{b_{ij}})\}\cr
		&=\sum_{j=1 }^d\{-\pa_{x_i}\partial^{\alpha}(\ell_{b_{jj}}+h_{b_{jj}})+\partial_{x_j}\partial^{\alpha}(\ell_{b_{ij}}+h_{b_{ij}})\} 
	\end{align*}
for $|\alpha|\le N-1$, where we used
	$$
 \partial_{x_{i}}b_{j}+\partial_{x_{j}}b_{i}=\ell_{b_{ij}}+h_{b_{ij}}.
$$
Multiplying it by $\pa^\alpha b_i$ and integrating over $x\in \T^d$, we get
	\begin{align*}
	\|\nabla_x\partial^{\alpha}b_{i}\|^{2}_{L^2_x}&\le C_{\varepsilon}\sum_{j=1}^d\big(\|\partial^{\alpha}\ell_{b_{ij}}\|_{L^2_x}^{2}+\|\partial^{\alpha}h_{b_{ij}}\|_{L^2_x}^{2}\big)+\varepsilon\sum_{j=1}^d\|\partial_{x_j}\partial^{\alpha}b_{i}\|^{2}_{L^2_x}
	\end{align*}
and thus in the presence of spatial derivatives, we obtain
	\begin{equation}\label{b spatial}
	\|\nabla_x\partial^{\alpha}b\|^{2}_{L^2_x}\le C\sum_{i,j=1}^d\big(\|\partial^{\alpha}\ell_{b_{ij}}\|^{2}_{L^2_x}+\|\partial^{\alpha}h_{b_{ij}}\|^{2}_{L^2_x}\big).
	\end{equation} 
The purely temporal derivative of $b(x,t)$ will be treated later.		For $ a(x,t) $, in the presence of the temporal derivative, using the first relation of the macro-micro decomposition gives
\begin{equation}\label{a temporal}
\|\partial_{t}\partial^{\alpha}a\|^{2}_{L^2_x}\le C\left(\|\partial^{\alpha}\ell_{a}\|^{2}_{L^2_x}+\|\partial^{\alpha}h_{a}\|^{2}_{L^2_x}\right).
\end{equation}
For purely spatial derivatives of $a(x,t)$ with $|\alpha|\neq 0$, we have
\begin{align*}
-\Delta_x \pa^\alpha a=\nabla_x\cdot\pa_t\pa^\alpha b-\sum_{i=1}^d\pa_{x_i}\pa^\alpha\{\ell_{a_i}+h_{a_i}\},
\end{align*}
which, combined with \eqref{b spatial}, leads to
\begin{align}\begin{split} \label{a spatial}
\|\nabla_x\partial^{\alpha}a\|^{2}&\le C  \|\pa_t \pa^\alpha b\|_{L^2_x}^2+ C\sum_{i=1}^d\big(\|\partial^{\alpha}\ell_{a_i}\|_{L^2_x}^{2}+\|\partial^{\alpha}h_{a_i}\|_{L^2_x}^{2}\big) \cr
&\le C  \sum_{i,j=1}^d\big(\|\pa_t\partial^{\alpha^\prime}\ell_{b_{ij}}\|^{2}_{L^2_x}+\pa_t\partial^{\alpha^\prime}h_{b_{ij}}\|^{2}_{L^2_x}+\|\partial^{\alpha}\ell_{a_i}\|_{L^2_x}^{2}+\|\partial^{\alpha}h_{a_i}\|_{L^2_x}^{2}\big) 
\end{split} \end{align} 
where $\pa^{\alpha^\prime}$ is the spatial derivative with $|\alpha^\prime|=|\alpha|-1$. For the case $|\alpha|=0$, we use the first relation of Proposition \ref{local} (3) and Poincar\'{e} inequality  to get
\begin{align}\label{a 0}
\| a\|^2_{L^2_x}&\le C\|\nabla_x a\|^2_{L^2_x}\le  C \|\pa_t   b\|_{L^2_x}^2+ C\sum_{i=1}^d\big(\| \ell_{a_i}\|_{L^2_x}^{2}+\| h_{a_i}\|_{L^2_x}^{2}\big). 
\end{align}
On the other hand, using the second relation of Proposition \ref{local} (3), Poincar\'{e} inequality and \eqref{b spatial} gives
$$
\| \pa_t b\|^2_{L^2_x}\le C\|\nabla_x \pa_t b\|^2_{L^2_x}\le C\sum_{i,j=1}^d\big(\|\partial_t\ell_{b_{ij}}\|^{2}_{L^2_x}+\|\partial_th_{b_{ij}}\|^{2}_{L^2_x}\big)
$$
which combined with \eqref{a 0} gives
\begin{align}\label{a}
\| a\|^2_{L^2_x}\le C\sum_{i,j=1}^d\big(\|\partial_t\ell_{b_{ij}}\|^{2}_{L^2_x}+\|\partial_th_{b_{ij}}\|^{2}_{L^2_x}+\|\ell_{a_i}\|_{L^2_x}^{2}+\| h_{a_i}\|_{L^2_x}^{2}\big).
\end{align}
Finally for the purely temporal derivative $\pa^\alpha=\pa_t^{\alpha_0}$, it follows from the first and second relations of the macro-micro decomposition that
\begin{align}\label{b temporal}\begin{split}
\| \pa_t\pa^\alpha b\|^2_{L^2_x}&\le C\| \nabla_x\pa^\alpha a\|^2_{L^2_x}+C\sum_{i=1}^d \big( \|\pa^\alpha\ell_{a_i}\|_{L^2_x}^{2}+\| \pa^\alpha h_{a_i}\|_{L^2_x}^{2}\big) \cr
&\le C\left(\|\nabla_x\partial_t^{\alpha_0-1}\ell_{a}\|^{2}_{L^2_x}+\|\nabla_x\partial_t^{\alpha_0-1}h_{a}\|^{2}_{L^2_x}\right)+C\sum_{i=1}^d \big( \|\pa^\alpha\ell_{a_i}\|_{L^2_x}^{2}+\| \pa^\alpha h_{a_i}\|_{L^2_x}^{2}\big) .
\end{split}\end{align}
Combining  \eqref{b spatial}, \eqref{a temporal}, \eqref{a spatial}, \eqref{a} and \eqref{b temporal} gives the desired result.
\end{proof}

\begin{lemma}\label{ell h} Let $f$ be a local solution constructed in Proposition \ref{local}.	We then have
\begin{enumerate}
	\item $\displaystyle \sum_{|\alpha|\le N-1}\big(\|\pa^\alpha\ell_a\|_{L^2_x}+\sum_{i=1}^d\|\pa^\alpha\ell_{a_i}\|_{L^2_x}+\sum_{i,j=1}^d\|\pa^\alpha\ell_{b_{ij}}\|_{L^2_x}\big)\le C\sum_{|\alpha|\le N}\|\{I-P\}\pa^\alpha f\|_{L^2_{x,v}},$
	\item $\displaystyle \sum_{|\alpha|\le N-1}\big(\|\pa^\alpha h_a\|_{L^2_x}+\sum_{i=1}^d\|\pa^\alpha h_{a_i}\|_{L^2_x}+\sum_{i,j=1}^d\|\pa^\alpha h_{b_{ij}}\|_{L^2_x}\big)\le C\sqrt{E(f) }\sum_{|\alpha|\le N} \|\pa^\alpha f\|_{L^2_{x,v}}$
\end{enumerate}

\end{lemma}
\begin{proof}
 By definition,	$\pa^\alpha\ell_{(\cdot)}$ can be written in the form of
	\begin{align*}
\lal \pa^\alpha\ell\{I-P\}(f) , e_{(\cdot)}\ral_{L^2_v}&=\intr \{-\partial_t -v\cdot\nabla_x+L \}\{I-P \}(\pa^\alpha f) e_{(\cdot)}\,dv\cr
&=\intr \{-\{I-P \}(\partial_t\pa^\alpha f) -v\cdot\{I-P \}(\nabla_x \pa^\alpha f)-\{I-P \}(\pa^\alpha f) \} e_{(\cdot)}\,dv
	\end{align*}
where we used Proposition \ref{linear} (1). Applying H\"{o}lder's inequality, we get 
	\begin{align*}
	\left\|\lal \pa^\alpha\ell\{I-P\}(f) , e_{(\cdot)}\ral_{L^2_v} \right\|_{L^2_x}&\le 	C\left(\left\|\{I-P\}(\pa_t\pa^\alpha f) \right\|_{L^2_{x,v}}+\left\|\{I-P\}(\nabla_x\pa^\alpha f) \right\|_{L^2_{x,v}}+\left\|\{I-P\}(\pa^\alpha f) \right\|_{L^2_{x,v}}\right).
	\end{align*}
Similarly, $\pa^\alpha h_{(\cdot)}$ can be  written as
$$
\pa^\alpha h_{(\cdot)}=\lal \pa^\alpha \Gamma(f) , e_{(\cdot)}\ral_{L^2_v}.
$$
Applying Lemma \ref{nonlinear} (2) and $H^2(\T^3)\subseteq L^\infty(\T^3)$ to the lower order derivative, we obtain
	\begin{align*}
\left\|\pa^\alpha h_{(\cdot)}\right\|_{L^2_x}\le C\sum_{|\beta|\le |\alpha|}\left\| \|\pa^\beta f\|_{L^2_v}\|\pa^{\alpha-\beta}f\|_{L^2_v} \right\|_{L^2_x} &\le C\sqrt{E(f)}\sum_{|\alpha|\le N} \|\pa^\alpha f\|_{L^2_{x,v}}
\end{align*}
which completes the proof.
\end{proof}
\begin{theorem}\label{L coer}
Let $f$ be a local solution constructed in Proposition \ref{local}.	Then there exists a positive constant $\delta$  such that if
	$	E(f)(t)\le \eta$, 	then
	$$
	\sum_{|\alpha|\le N} \lal L(\pa^\alpha f),\pa^\alpha f\ral_{L^2_{x,v}}\le -\delta  \sum_{ |\alpha|\le N} \|\pa^\alpha f\|^2_{L^2_{x,v}}.
	$$
\end{theorem}
\begin{proof}
	Combining Lemma \ref{ab estimates} and Lemma \ref{ell h}, we get
	\begin{align*}
\sum_{|\alpha|\le N} \| P(\pa^\alpha f)\|^2_{L^2_{x,v}}&\le C\sum_{|\alpha|\le N}\| \big( \{I-P\}\pa^\alpha f\|^2_{L^2_{x,v}}+ \sqrt{E(f)}\| \pa^\alpha f\|^2_{L^2_{x,v}} \big).
\end{align*}
Splitting the last term into the macroscopic part $P$ and the microscopic part $\{I-P\}$, we obtain
	\begin{align}\label{P I-P}
C_1\sum_{|\alpha|\le N} \| P(\pa^\alpha f)\|^2_{L^2_{x,v}}&\le \sum_{|\alpha|\le N}\|  \{I-P\}\pa^\alpha f\|^2_{L^2_{x,v}} 
\end{align}
for sufficiently small $E(f)$. Thus it follows from Proposition \ref{linear} (2) and \eqref{P I-P} that
		\begin{align*}
	\sum_{|\alpha|\le N} \lal L(\pa^\alpha f),\pa^\alpha f\ral_{L^2_{x,v}}&=-\sum_{|\alpha|\le N}\|\{I-P\}\pa^\alpha f\|^2_{L^2_{x,v}}\cr 
	&\le -\frac 12\sum_{|\alpha|\le N}\|\{I-P\}\pa^\alpha f\|^2_{L^2_{x,v}}-\frac {C_1}{2}\sum_{|\alpha|\le N}\|P(\pa^\alpha f)\|^2_{L^2_{x,v}}\cr 
	&\le -\frac{1}{2}\min \{1,  C_1\} \sum_{ |\alpha|\le N} \| \pa^\alpha f\|^2_{L^2_{x,v}},
\end{align*}
which completes the proof.
\end{proof}
\subsection{Global existence and large-time behavior} In this subsection, we extend the local-in-time existence of solutions established in Proposition \ref{local} to the global one via the standard continuity argument and show the large-time behavior of solutions. For this, we apply $\partial^\alpha$ to \eqref{linearized Cauchy}, multiply by $\partial^{\alpha}f$, integrate over $(x,v)\in\T^3\times\R^3$, and sum over $|\alpha|\le N$ to obtain
\begin{align*} 
\frac{1}{2}\frac{d}{dt}E(f)(t)+\delta E(f)(t)&\le \sum_{|\alpha|\le N}\lal  \pa^\alpha\Gamma(f),\pa^\alpha f\ral_{L^2_{x,v}}\cr
&\le C\sum_{\substack{|\beta|\le |\alpha|\cr |\alpha|\le N}}\int_{\T^d}\|\pa^{\beta}f\|_{L^2_{v}}\|\pa^{\alpha-\beta}f\|_{L^2_{v}}\|\pa^\alpha f\|_{L^2_{v}}\,dx
 \end{align*}
where we used Lemma \ref{nonlinear} and  Theorem \ref{L coer}. For the right-hand side of the last inequality, we apply $H^2(\T^3)\subseteq L^\infty(\T^3)$ to the lower order derivative  to get
\begin{align*} 
\frac{d}{dt}E(f)(t)+\delta E(f)(t)\le C\sqrt{E(f)(t)}E(f)(t).
\end{align*}
Define
$$
M=\min\left\{\frac{\delta^2}{4C^2},\eta\right\}
$$
and choose the initial data sufficiently small in the sense that
$$
E(f_{0})\le \frac M2 < \eta.
$$
Let $T>0$ be given as
$$
T=\sup_{t}\big\{t:E(f)(t)\le M\big\}
$$
which gives
$$
E(f)(t)\le M \le \eta.
$$
 For $0\le t\le T$, we have
\begin{align}\label{asymptotic}
\frac{d}{dt}E(f)+\delta E(f)&\le C \sqrt{E(f)}E(f)\le C \sqrt{M}E(f)\le \frac{\delta}{2} E(f),
\end{align}
yielding
\begin{align*}
E(f)(t)+\frac{\delta}{2}\int_{0}^{t}E(f)(s)ds&\le E(f_{0})\le \frac M2 < M,
\end{align*}
which is a contradiction to the continuity of $E(f)(t)$. Therefore by definition of $T$, we conclude that $T=\infty$ and applying Gr\"{o}nwall's inequality to  \eqref{asymptotic} gives the desired result.

\section*{Acknowledgments}
 B.-H. Hwang was funded by a 2023 Research Grant from Sangmyung University (2023-A000-0284).

%%%%%%%%%%%%%%%%%%%%%%%%%%%%%%%%%%%%%%%%%%%%%%%%
%
%
%                        thebibliography
%
%
%%%%%%%%%%%%%%%%%%%%%%%%%%%%%%%%%%%%%%%%%%%%%%%%%%%%%%%%%%%%%%%%%%%%%%%%%%%%%%%%%

\end{document}